\documentclass[11pt]{article}

\usepackage{amsmath}
\usepackage{latexsym}
\usepackage{amsfonts}
\usepackage{amssymb}
\usepackage{amscd}
\usepackage{array}
\usepackage{color}
\usepackage{verbatim}
\usepackage{theorem}
\newtheorem{theorem}{Theorem}[section]

\newtheorem{lemma}[theorem]{Lemma}

\newtheorem{corollary}[theorem]{Corollary}

{\theorembodyfont{\rmfamily}

  \newtheorem{remark}[theorem]{Remark}
}

\newenvironment{proof}{    % De proof o de prueba seg\'{u}n convenga
  \noindent
  \textbf{Proof.}}{
  \hfill $\Box$
  \vspace{3mm}
}

\numberwithin{equation}{section}

\newcommand{\N}{\mathbb{N}} %% Conjunto naturales:     \N
\newcommand{\C}{\mathbb{C}} %% Conjunto complejos:     \C
\newcommand{\D}{\mathbb{D}} %% Disco unidad:           \D

\newcommand{\n}{|\!|\!|}

\DeclareMathOperator*{\ind}{ind}    %% L\'{\i}mite inductivo    \indlim
\begin{document}

\title{Invariant subspaces of the integration operators on H\"ormander algebras and Korenblum type spaces}

\author{Jos\'{e} Bonet and Antonio Galbis }

\date{}

\maketitle

\begin{abstract}
We describe the proper closed invariant subspaces of the integration operator when it acts continuously on countable
intersections and countable unions of weighted Banach spaces of holomorphic functions on the unit disc or the complex plane.
Applications are given to the case of Korenblum type spaces and H\"ormander algebras of entire functions.
\end{abstract}

%% Footnotes
\renewcommand{\thefootnote}{}
\footnotetext{This article is an improved version of the paper \cite{B}, which appeared in Arxiv, reference 2003.13573}

%%%%%%%%%%%%%%%%%%%%%%%%%%%%%%%%%%%%%%%%%%%%%%%%%%%%%%%%%%%%%%%%%%%%%%%%%
%%%%%%%%%% AQU\'{I} EMPIEZA EL DOCUMENTO
%%%%%%%%%%%%%%%%%%%%%%%%%%%%%%%%%%%%%%%%%%%%%%%%%%%%%%%%%%%%%%%%%%%%%%%%%
%%%%%%%%%%%%%%%%%%%%%%%%%%%%%%%%%%%%%%%%%%%%%%%%%%%%%%%%%%%%%%%%%%%%%%%

\section{Introduction.}

Let $G$ be the open unit disc $\D$ or the whole complex plane $\C$. We denote by $H(G)$ the Fr\'echet space of holomorphic functions on $G$, endowed with the topology of uniform convergence on compact subsets of $G$. A \textit{space $E$ of holomorphic functions on the domain $G$} is a Hausdorff locally convex space that is a subset of $H(G)$, such that the inclusion map $E \subset H(G)$ is continuous and $E$ contains the polynomials. If $E$ is a Fr\'echet space or a countable inductive limit of Banach spaces, by the closed graph theorem, the inclusion map $E \subset H(G)$ is continuous if and only if the point evaluations $E \rightarrow \C, f \rightarrow f(z)$ at all the points $z \in G$ are continuous on $E$. We are mainly interested in the case when the polynomials are dense in $E$. In this case then $E$ is separable.

Banach spaces of holomorphic functions on the unit disc $\D$ and on the complex plane $\C$ have been thoroughly investigated. We refer the reader for example to the books \cite{HKZ}, \cite{Z} and \cite{Z1}. H\"ormander algebras of entire functions \cite{A}, \cite{BG}, \cite{BLV}, \cite{BF}, \cite{MT}, as well as Koremblum space and other intersections and unions of growth Banach spaces of holomorphic functions on the unit disc \cite{ABR}, \cite{BLT}, \cite{HKZ}, are natural examples of (locally convex) spaces  of holomorphic functions. Vogt \cite{V13} proved that there are Fr\'echet spaces $E$ which are contained in $H(G)$ such that the inclusion $E \subset H(G)$ is not continuous.

In this note we study the set of proper closed invariant subspaces of the integration operator

$$
Jf(z):= \int_0^z f(\zeta) d \zeta, \ \ z \in G, \ \ f \in H(G),
$$
\noindent
when it acts continuously on Fr\'echet spaces or countable inductive limits of Banach spaces (called (LB)-spaces) $E$, which appear as countable intersections or unions of weighted Banach spaces of holomorphic functions on the unit disc or the complex plane. Recall that a subspace $M$ of a locally convex space $E$ is called invariant of a continuous linear operator $T:E \rightarrow E$ if $T(M) \subset M$. Our main results are Theorem \ref{invsubsdisc} and its Corollary \ref{koren}, which describe the invariant subspaces when the integration operator acts on Korenblum type spaces, and Theorem \ref{theorem_Ha}, which explains the situation in case of some H\"ormander algebras of entire functions. The proofs of these results depend on some abstract Theorems \ref{theorem_F} and \ref{theorem_LB} and they rely heavily on Theorems 3.8 and 3.16 due to Abanin and Tien \cite{AT17-2}. A different method permits us to handle the (LB)-algebra of entire functions of exponential type in Theorem \ref{exp}. An open question about the invariant subspaces of the integration operator on certain H\"ormander algebras is mentioned in Remark \ref{question}.

Abanin and Tien describe in \cite{AT17-2} the closed invariant subspaces of the integration operator on various scales of weighted Banach spaces of holomorphic functions. As mentioned above, some of their results are very important for our theorems below. We refer the reader to the introduction of \cite{AT17-2} for classical results about invariant subspaces of the integration operators and more recent ones in \cite{AK}, \cite{C} and \cite{CP}. The continuity of the integration operator on weighted Banach spaces of holomorphic functions was investigated by Harutyunyan and Lusky \cite{HL}; see also \cite{AT15} and \cite{AT18}. Other aspects, like spectrum and ergodic or dynamical properties, were considered by Beltr\'an, Fern\'andez and the first author in \cite{BBF13}. Similar questions for operators defined on H\"ormander algebras were investigated in \cite{BBF15}.

Our notation for functional analysis, in particular for locally convex spaces,  Fr\'echet spaces and (LB)-spaces, is standard. We refer the reader to \cite{BB}, \cite{J}, \cite{MV} and \cite{T}. If $E$ is a locally convex space, its topological dual is denoted by $E'$. The linear span of a subset $A$ of $E$ is denoted by ${\rm span}(A)$. The closure of a subset $A$ in $E$ is denoted by $\overline{A}$, and $\overline{A}^{E}$ in case the space in which the closure is taken must be emphasized. A subspace $M$ of $E$ is called \textit{proper} if $\{ 0 \} \neq M \neq E$.
In what follows, we set $\N_0 := \N \cup \{ 0 \}$.

\section{Results about subspaces of Fr\'echet or (LB)-spaces.}

\begin{lemma}\label{lemma_abstract_F}
Let  $ X := \mbox{\rm proj}_n X_n $ be a Fr\'echet space such that $ X = \bigcap_{ n \in \N} X_n  $ with each $ ( X_n, \| \cdot \| _n ) $
 a Banach space. Moreover, it is assumed that $ X $ is dense
 in $ X_n $ and that  $ X _{ n + 1} \subseteq X_n $ with a continuous inclusion for each $ n \in \N $. Let $M$ be a subspace of $X$. Then
\begin{itemize}
\item[(i)] $\overline{M}^{X} = \bigcap_{n \geq 1} \overline{M}^{X_n}$.
\item[(ii)] If $M$ is proper and closed, then there is $n(0) \in \N$ such that $\overline{M}^{X_n}$ is proper in $X_n$ for each $n \geq n(0)$.
\end{itemize}
\end{lemma}
\begin{proof}
(i) Since the inclusions $X \subset X_{n+1} \subset X_n$ are continuous for each $n \in \N$, we clearly have $\overline{M}^{X} \subset \bigcap_{n \geq m} \overline{M}^{X_n} = \bigcap_{n \geq 1} \overline{M}^{X_n}$ for each $m \in \N$.
Now, given $x \in \bigcap_{n \geq 1} \overline{M}^{X_n}$, we have $x \in \bigcap_{n \geq 1} X_n = X$. Moreover, for each $n \in \N$ there is $y(n) \in M$ such that
$\|x-y(n)\|_n < 1/n$. Then $x = \lim_{n \rightarrow \infty} y(n)$ in $X$ and $x \in \overline{M}^{X}$.

(ii) First of all, since $\{ 0 \} \neq M$, we get $\{ 0 \} \neq \overline{M}^{X_n}$ for each $n \in \N$. Proceeding by contradiction, suppose that there is an increasing sequence $(n(k))_k$ of natural numbers such that $\overline{M}^{X_{n(k)}} = X_{n(k)}$ for each $k \in \N$. Since $ X := \mbox{\rm proj}_k X_{n(k)} $, we could apply part (i) to conclude
$\overline{M}^{X} = X$. Since $M$ is closed by assumption, we have $M=X$, and $M$ would not be a proper subspace.
\end{proof}

\begin{lemma}\label{lemma_abstract_LB}
Let $ X = \mbox{\rm ind}_n X_n $ be an (LB)-space with increasing union of Banach spaces  $ X = \bigcup_{ n \in \N} X_n $. If $M$ is a proper closed
subspace of $X$, then there is $n(0) \in \N$ such that $M \cap X_n$ is a proper closed subspace of $X_n$ for each $n \geq n(0)$.
\end{lemma}
\begin{proof}
Since the inclusion $X_n \subset X$ is continuous for each $n \in \N$, the set $M \cap X_n$ is closed in $X_n$ for each $n \in \N$.
As $M$ is proper, there is $x \in M, \ x \neq 0$. There is $n(1) \in \N$ such that $x \in M \cap X_{n(1)}$. On the other hand, since $M \neq X$, there is $y \in X \setminus M$.
Select $n(0) \geq n(1)$ such that $y \in X_{n(0)}$. Then, if $n \geq n(0)$, we have $x \in M \cap X_{n}$ and $y \in X_n \setminus M$. Thus, $M \cap X_n$ is a proper subspace of $X_n$
\end{proof}

\section{Abstract results about invariant subspaces of the integration operator on  spaces of holomorphic functions.}

Our first Lemma is stated in page 412 of \cite{AT17-2}. We include a proof for the sake of completeness, because it is very relevant in our considerations below. As in \cite{AT17-2}, given a space of holomorphic functions $E$ on the open unit disc $\D$ or the complex plane $\C$, and $N \in \N$, we set
$$
A^{E}_N:=\{f \in E \ | \ f^{(j)}(0)=0, 0 \leq j < N \}.
$$

\begin{lemma}\label{lemma_form_subspaces}
Let $E$ be a space of holomorphic functions on the open domain $G=\D$ or  $G=\C$, such that the polynomials are dense in $E$. For each $N \in \N$ we have
$$
A^{E}_N = \overline{{\rm span}(\{z^j \ | \ j \geq N \})}.
$$
\end{lemma}
\begin{proof}
Fix $n \in \N$. Since the inclusion $E\subset H(G)$ is continuous, the evaluations $E \rightarrow \C, f \rightarrow f^{(j)}(0)$ are continuous for each $j \in \N$ and the set $A^{E}_N$ is closed. Moreover, it clearly contains $z^j$ for each $j \geq N$. Hence, $\overline{{\rm span}(\{z^j \ | \ j \geq N \})} \subset A^{E}_N$.

The Taylor expansion at the origin of a holomorphic function $f \in H(G)$ is denoted by $f(z) = \sum_{j=0}^{\infty} a_j(f) z^j$, for each $z \in G$.
Fix $f \in A^{E}_N.$ Since the polynomials are dense in $E$ by assumption, there is
a sequence of polynomials $(g_k)_k$ such that $\lim_{k \rightarrow \infty} g_k = f$ in $E$. This implies that $\lim_{k \rightarrow \infty} g_k = f$ in $H(G)$.
Therefore, $\lim_{k \rightarrow \infty} \sum_{j=0}^{N-1} a_j(g_k) z^j = 0$ in $H(G)$. The span of $\{1,z,...,z^{N-1} \}$ is finite dimensional, hence
$\lim_{k \rightarrow \infty} \sum_{j=0}^{N-1} a_j(g_k) z^j = 0$ in $E$. Consequently, $(g_k - \sum_{j=0}^{N-1} a_j(g_k) z^j)_k$ is a sequence of elements of
${\rm span}(\{z^j \ | \ j \geq N \})$ which converges in the topology of $E$ to $f$. Thus, $f \in \overline{{\rm span}(\{z^j \ | \ j \geq N \})}$.
\end{proof}

\begin{theorem}\label{theorem_F}
Let $ F := \mbox{\rm proj}_n F_n $ be a Fr\'echet space such that $ F = \bigcap_{ n \in \N} F_n  $ with each $F_n$ a Banach space of holomorphic functions
on the open domain $G=\D$ or  $G=\C$, such that the polynomials are contained in $F$ and dense in $F_n$ for each $n \in \N$.

Assume that the integration operator $J:F_n \rightarrow F_n$ is continuous for each $n \in \N$, and that, for each $n \in \N$, every proper closed
invariant subspace of $J$ on $F_n$ is of the form $A^{F_n}_K$ for some $K \in \N$. Then
\begin{itemize}
\item[(i)] $J:F \rightarrow F$ is continuous, and
\item[(ii)] Every proper closed invariant subspace for $J$ on $F$ is of the form $$A^{F}_K = \{f \in F \ | \ f^{(j)}(0)=0, 0 \leq j < K \}$$ for some $K \in \N$.
\end{itemize}
\end{theorem}
\begin{proof}
(i) The continuity of $J:F_n \rightarrow F_n$ for each $n \in \N$ clearly implies  $J(F) \subset F$ and that $J:F \rightarrow F$ is continuous, because $F$ is the projective limit of the sequence of Banach spaces $(F_n)_n$.

(ii) It is easy to see that  $A^{F}_K$ is a closed invariant subspace of the integration operator $J$ on $F$; see for example Lemma \ref{lemma_form_subspaces}. Since $1 \notin A^{F}_K$ and $z^K \in A^{F}_K$, it follows
$A^{F}_K$ is a proper subspace of $E$.

Now, let $M$ be a proper closed subspace of $F$ which is invariant for the integration operator $J$. By Lemma \ref{lemma_abstract_F} (ii) there is $n(0) \in \N$ such that
$\overline{M}^{F_n}$ is proper (and closed) in $F_n$ for each $n \geq n(0)$. Moreover, it is invariant for $J$, since
$$J(\overline{M}^{F_n}) \subset \overline{J(M)}^{F_n} \subset \overline{M}^{F_n}.$$

By assumption, for each $n \geq n(0)$, there is $k(n) \in \N$ such that
$$
\overline{M}^{F_n} = A^{F_n}_{k(n)} = \overline{{\rm span}(\{z^j \ | \ j \geq k(n) \})}^{F_n}.
$$
Take, for each $n \in \N$, $k(n)$ as the minimum of those $j's$ such that $z^j \in \overline{M}^{F_n}$. Clearly $1 \leq k(n)$ for each $n \geq n(0)$. This selection implies, in particular,  that $z^{k(n)-1} \notin \overline{M}^{F_n}$. Since $\overline{M}^{F_{n+1}} \subset \overline{M}^{F_n}$ for each $n \in \N$, we have $1 \leq k(n(0)) \leq k(n) \leq k(n+1)$ for each $n \geq n(0)$.

We claim that the sequence $(k(n))_n$ of natural numbers must be bounded. If this is not the case, $\lim_{n \rightarrow \infty} k(n) = \infty$. For each $n \geq n(0)$ we have
$$
M \subset \overline{M}^{F_n} = \overline{{\rm span}(\{z^j \ | \ j \geq k(n) \})}^{F_n} \subset \{f \in H(G) \ | \ f^{(j)}(0)=0, 0 \leq j < k(n) \}.
$$
Therefore, the Taylor coefficients $(a_j(f))_{j=0}^{\infty}$ of each $f \in M$ must vanish; that is, $M=\{ 0 \}$, which is a contradiction, as $M$ is proper.

As a consequence of the claim that we have just proved, there are $K \in \N$ and  $n(1) \geq n(0)$ such that $k(n)=K$ for each $n \geq n(1)$. This implies
$$
\overline{M}^{F_n} = \{f \in F_n \ | \ f^{(j)}(0)=0, 0 \leq j < K \}
$$
for each $n \geq n(1)$. Thus, we may apply Lemma \ref{lemma_abstract_F} (i) to get
$$
M = \bigcap_{n \geq n(1)} \overline{M}^{F_n} = \{f \in F \ | \ f^{(j)}(0)=0, 0 \leq j < K \}.
$$
This completes the proof.
\end{proof}

\begin{theorem}\label{theorem_LB}
Let $ E = \mbox{\rm ind}_n E_n $ be an (LB)-space which is the increasing union of Banach spaces  $ E_n $ of holomorphic functions
on the open domain $G=\D$ or  $G=\C$, such that the polynomials are contained and dense in $E_n$ for each $n \in \N$.

Assume that the integration operator $J:E_n \rightarrow E_n$ is continuous for each $n \in \N$, and that, for each $n \in \N$, every proper closed
invariant subspace of $J$ on $E_n$ is of the form $A^{E_n}_K$ for some $K \in \N$. Then
\begin{itemize}
\item[(i)] $J:E \rightarrow E$ is continuous, and
\item[(ii)] Every proper closed invariant subspace of $J$ on $E$ is of the form $$A^{E}_K = \{f \in E \ | \ f^{(j)}(0)=0, 0 \leq j < K \}$$ for some $K \in \N$.
\end{itemize}
\end{theorem}
\begin{proof}
(i) Clearly $J(E) \subset E$ and $J:E \rightarrow E$ is continuous since $ E = \mbox{\rm ind}_n E_n $ and $J$ is stepwise continuous by assumption.

(ii) The assumptions imply that the polynomials are dense in $E$. The subspaces $A^{E}_K$ are proper closed subspaces of $E$ which are invariant for the integration operator $J$.

We fix a proper closed subspace $M$ of $E$ which is invariant for $J$. By Lemma \ref{lemma_abstract_LB} there is $n(0) \in \N$ such that
$M \cap E_n$ is a proper closed subspace of $E_n$ for each $n \geq n(0)$. It is also invariant for $J$, since $J(M) \subset M$ and $J(E_n) \subset E_n$.
By our assumptions, for each $n \geq n(0)$ there is $k(n) \in \N$ such that
$$
M \cap E_n = A^{E_n}_{k(n)} = \overline{{\rm span}(\{z^j \ | \ j \geq k(n) \})}^{F_n}.
$$
We select, for each $n \in \N$, $k(n)$ as the minimum of those $j's$ such that $z^j \in M \cap E_n$. Since $M \cap E_n \subset M \cap E_{n+1}$ for each $n \in \N$, we have
$1 \leq k(n+1) \leq k(n) \leq k(n(0))$ for each $n \geq n(0)$. Then there are $L \in \N$ and $n(1) \geq n(0)$ such that $k(n)=L$ for each $n \geq n(1)$. This yields, for each $n \geq n(1)$,
$$
M \cap E_n = \{f \in E_n \ | \ f^{(j)}(0)=0, 0 \leq j < L \}.
$$
Therefore
$$
M = \bigcup_{n \geq n(1)} (M \cap E_n) = \{f \in E \ | \ f^{(j)}(0)=0, 0 \leq j < L \},
$$
and the proof is complete.
\end{proof}

\section{Invariant subspaces of the integration operator on Fr\'echet or (LB)-spaces of holomorphic functions on the disc. }\label{sectdisc}

Let us introduce some notation. We set $R= 1$ (for the case of holomorphic functions on $\D$) and $R= +\infty$ (for the case of entire functions).  A {\it weight} $v$   is a continuous function  $v: [0, R[ \to ]0,  \infty [$, which is decreasing on $[0,R[$ and satisfies $\lim_{r \rightarrow R} r^n v(r)=0$ for each $n \in \N$. We extend $v$ radially to $\D$ if $R=1$ and to $\C$ if $R=+\infty$ by $v(z):= v(|z|)$. For
such a weight $v$, we define the Banach space  $H_v^\infty$ of holomorphic
functions $f$ on the disc $\D$ (if $R=1$) or on the whole complex plane $\C$ (if $R=+\infty$) such that $\Vert f \Vert_v:= \sup_{|z| <R} v(z)|f(z)| <\infty$. For a holomorphic
function $f \in H(\{z \in \C ; |z| <R \})$ and $r <R$, we denote $M(f,r):= \max\{|f(z)| \ ; \ |z|=r\}$. Using the notation $O$ and $o$ of Landau, $f \in H_v^\infty$ if and only if $M(f,r)=O(1/v(r)), r \rightarrow R$. It is known that the closure of the polynomials in $H_v^\infty$ coincides with the Banach space $H_v^0$ of all those holomorphic functions on $\{z \in \C ; |z| <R \}$ such that $M(f,r)=o(1/v(r)), r \rightarrow R$, see e.g.\ \cite{BBG}. Spaces of type $H_v^\infty$ appear in the study of growth conditions of
holomorphic functions and have been investigated in various articles since the work of Shields and Williams, see e.g.\ \cite{BBG}, \cite{L2} and the references therein.
When we must specify the domain of definition of the holomorphic functions, we write $H_v^0(\D)$ or $H_v^0(\C)$.

We recall some examples of weights: For $R=1$,  $(i)$ $v_{\alpha}(r)=(1-r)^{\alpha}$ with $\alpha >0$, which are the \textit{standard weights} on the disc,
$(ii)$ $v(r)= \exp(-(1-r)^{-1})$, and $(iii)$ $v(z)=( \log	\frac{e}{1-r})^{-\alpha}, \ \alpha>0$. For $R=+\infty$, $(i)$ $v(r)= \exp(-r^p)$ with  $p > 0$,
$(ii)$ $v(r)= \exp( -\exp r)$,  and $(iii)$ $v(r)= \exp\big(- (\log^+ r  )^p\big) $, where $p \geq 2$ and $\log^+ r = \max(\log r,0)  $.

A systematic study of inductive and projective limits of weighted Banach spaces of type $H_v^\infty$ and $H_v^0$ was initiated by Bierstedt, Meise and Summers in \cite{BMS}.
If $\mathcal{V} = (v_n)_n$ is a decreasing sequence of weights on $G=\D$ or $G=\C$, we define the weighted inductive limit by $\mathcal{V}_0 H(G):=  \mbox{\rm ind}_n H^0_{v_n}(G)$.
These (LB)-spaces have been investigated by many authors; see e.g.\ \cite{AT18}, \cite{BB06} and \cite{BBG} and the references therein. On the other hand, if $\mathcal{A} = (a_n)_n$ is an increasing sequence of weights on $G=\D$ or $G=\C$, the weighted Fr\'echet space of holomorphic functions on $G$ is defined by $\mathcal{A}_0 H(G):=  \mbox{\rm proj}_n H^0_{a_n}(G)$. Fr\'echet spaces of this type were studied by E. Wolf \cite{W}. Concrete examples of spaces of this type in the case of entire functions appear in Section \ref{sectHoer}. The most relevant examples in the case of holomorphic functions on the disc are Koremblum type spaces, which we define now.

For each $\mu > 0$,  the growth Banach spaces of holomorphic functions are defined as $A^{-\mu}:=H^{\infty}_{v_{\mu}}$ and $A_0^{-\mu}:=H^{0}_{v_{\mu}}$ for the standard weight $v_{\mu}(r)=(1-r)^{\mu}$. These Banach  spaces play a relevant role in connection with interpolation and sampling of holomorphic functions;  see \cite[Section 4.3]{HKZ}. The \textit{spaces of holomorphic functions of Korenblum type} are defined as follows.

\begin{eqnarray*}
&& A_+^{-\gamma}:=\cap_{\mu>\gamma}A^{-\mu}=\{f\in H(\D)\colon \sup_{z\in \D}(1-|z|)^\mu|f(z)|<\infty\ \forall \mu>\gamma\}	,
\end{eqnarray*}
in which case also
\begin{eqnarray*}
&& A_+^{-\gamma} = \cap_{\mu>\gamma} A_0^{-\mu}.
\end{eqnarray*}
for each $\gamma \ge 0 $. And

\begin{eqnarray*}
 && A_-^{-\gamma}  :=  \cup_{\mu<\gamma}A^{-\mu}=\{f\in H(\D)\colon  \sup_{z\in \D}(1-|z|)^\mu|f(z)|<\infty \mbox{ for some } \mu < \gamma\}	
\end{eqnarray*}
in which case also
\begin{eqnarray*}
 A_-^{-\gamma}=\cup_{\mu<\gamma} A_0^{-\mu},
\end{eqnarray*}
for each $0<\gamma \le \infty $.

The space  $A_+^{-\gamma}$ is a Fr\'echet space when  endowed with the locally convex topology  generated by the  increasing
sequence of norms $\n  f \n_k:=\sup_{z\in \D}(1-|z|)^{\gamma+\frac{1}{k}}|f(z)|$, for $f \in A_+^{-\gamma}$ and each $k \in \N $.
Clearly, $A_+^{-\gamma}$ is a Fr\'echet space of holomorphic functions of type $\mathcal{A}_0 H(\D)$ for an increasing sequence $\mathcal{A}=(a_n)_n$.

Each space $ A_-^{-\gamma}$  is endowed with the finest locally convex topology such that all the natural inclusion maps
$A^{-\mu}\subset A_-^{-\gamma} $, for $ \mu<\gamma $,  are continuous. In particular,  $A_-^{-\gamma}:=\ind_k A^{-(\gamma -\frac{1}{k})}=\ind_k A_0^{-(\gamma -\frac{1}{k})}$
is an (LB)-space. Of course, the inductive limit is taken over all $k \in \N$ such that $(\gamma - \frac 1k) > 0 $.
The  \textit{Korenblum space} (see \cite{Kor}) $A_-^{-\infty}$, denoted $ A^{-\infty}$, is $A^{-\infty}=\ind_n A^{-n}$.
All these (LB)-spaces are weighted inductive limits of the form $\mathcal{V}_0 H(G)$ for a decreasing sequence $\mathcal{V}=(v_n)_n$.

\begin{theorem}\label{invsubsdisc}
(1) If $\mathcal{A}=(a_n)_n$ is an increasing sequence of weights on the unit disc $\D$, then the integration operator $J: \mathcal{A}_0 H(\D) \rightarrow \mathcal{A}_0 H(\D)$ is continuous and every proper closed invariant subspace of $J$ on the weighted Fr\'echet space $\mathcal{A}_0 H(\D)$ is of the form $$\{f \in \mathcal{A}_0 H(\D) \ | \ f^{(j)}(0)=0, 0 \leq j < K \}$$ for some $K \in \N$.

(2) If $\mathcal{V}=(v_n)_n$ is a decreasing sequence of weights on the unit disc $\D$, then the integration operator $J: \mathcal{V}_0 H(\D) \rightarrow \mathcal{V}_0 H(\D)$ is continuous and every proper closed invariant subspace of $J$ on the weighted (LB)-space $\mathcal{V}_0 H(\D)$ is of the form $$\{f \in \mathcal{V}_0 H(\D) \ | \ f^{(j)}(0)=0, 0 \leq j < K \}$$ for some $K \in \N$.
\end{theorem}
\begin{proof}
For each weight $v$ on the unit disc $\D$, it follows from  \cite[proposition 3.1 and Theorem 3.8]{AT17-2} (see also \cite{AT15}), that the integration operator $J:H_v^0(\D) \rightarrow H_v^0(\D)$ is continuous (even compact), and that every proper closed invariant subspace of $J$ on $H_v^0(\D)$ is of the form $$\{f \in H_v^0(\D) \ | \ f^{(j)}(0)=0, 0 \leq j < K \}$$ for some $K \in \N$.

(1) The integration operator on the Fr\'echet space $\mathcal{A}_0 H(\D) = \mbox{\rm proj}_n H^0_{a_n}(\D)$ satisfies all the hypothesis of Theorem \ref{theorem_F}. The conclusion follows from this result.

(2) Similarly, $J: \mathcal{V}_0 H(\D) \rightarrow \mathcal{V}_0 H(\D)$ satisfies the assumptions of Theorem \ref{theorem_LB}, which permits to complete the proof.

\end{proof}

We have the following consequence of Theorem \ref{invsubsdisc}.

\begin{corollary}\label{koren}
The integration operator $J$ is continuous on the Korenblum type spaces $E=A_+^{-\gamma}, \ \gamma \geq 0$ and $E= A_-^{-\gamma}, \ 0 < \gamma \leq \infty,$ and
every proper closed invariant subspace of $J$ on $E$ is of the form $A^{E}_K = \{f \in E \ | \ f^{(j)}(0)=0, 0 \leq j < K \}$ for some $K \in \N$.
\end{corollary}

\section{Invariant subspaces of the integration operator on H\"or\-man\-der algebras.}\label{sectHoer}

A function $p: \C \rightarrow ]0, \infty [$ is called a \textit{weight exponent function} if it satisfies the following properties: (w1) $p$ is continuous and subharmonic, (w2) $p$ is radial, $p(z)=p(|z|), z \in \C,$ (w3) $\log(1 + |z|^2) = o(p(z))$ as $|z| \rightarrow \infty$; and (w4) $p$ is doubling, i.e.\ $p(2z)=O(p(z))$ as $|z| \rightarrow \infty$.

Given a weight $p$, we define the following weighted spaces of entire functions:

$$A_p(\C):=\{f\in {\mathcal H}(\C): \ {\rm there \ is} \ A>0: \sup_{z\in \C}|f(z)| \exp(-Ap(z))<\infty \},$$

\noindent
endowed with the inductive limit topology, for which it is an (LB)-space  and

$$A^0_p(\C):=\{f\in {\mathcal H}(\C): \ {\rm for \ all} \ \varepsilon>0: \sup_{z\in \C}|f(z)| \exp(-\varepsilon p(z))<\infty \},$$

\noindent
endowed with the projective limit topology, for which it is a Fr\'echet space.

These spaces are topological algebras. They are called H\"ormander algebras of entire functions. Clearly $A^0_p(\C) \subset A_p(\C)$. Condition (w3) implies that $A^0_p(\C)$ contains the polynomials, and condition (w4) implies that the spaces are stable under differentiation. Weighted algebras of entire functions of this type have been considered by many authors. See e.g.\ \cite{A}, \cite{BG}, \cite{BLV}, \cite{M} and \cite{MT} and the references therein.

Here are some examples: When $p(z)=|z|^s, \ s>0$, then $A_p(\C)$ consists of all entire functions of order $s$ and finite type or order less than $s$; and $A^0_p(\C)$ is the space of all entire functions of order at most $s$ and type $0$. For $s=1$, $A_p(\C)$ is the space of all entire functions of exponential type $Exp(\C)$, and $A^0_p(\C)$ is the space of  entire functions of infraexponential type.

It was proved by Beltr\'an, Fern\'andez and the first author in \cite[Lemma 4.1]{BBF15} that
the integration operator $J$ is continuous on $E=A_p(\C)$ and $E=A^0_p(\C)$ for every weight exponent $p$. Concerning invariant subspaces we have the following result.

\begin{theorem}\label{theorem_Ha}
Let $p$ be a differentiable weight exponent such that $\lim_{r \rightarrow \infty} p'(r) = \infty$. Then every proper closed invariant subspace of $J$ on $E=A^0_p(\C)$ or $E=A_p(\C)$ is of the form $$A^{E}_K = \{f \in E \ | \ f^{(j)}(0)=0, 0 \leq j < K \}$$ for some $K \in \N$.
This holds in particular for $p(z)=\alpha |z|^s, \ \alpha >0, s >1$, and for $p(z)=e^{|z|}$.
\end{theorem}
\begin{proof}
Define, for $\beta >0$, the weight $v_{\beta}(z):= \exp(-\beta p(z))$. The polynomials are dense in $H^0_{v_{\beta}}$ and, by \cite[Theorem 3.8]{AT17-1}, the  operator
$J: H^0_{v_{\beta}} \rightarrow H^0_{v_{\beta}}$ is continuous. Indeed,  $w_{\beta}(r):=1/v_{\beta}(r)=\exp(\beta p(r)), \ r \geq 0,$ satisfies
$$
\lim {\rm inf}_{r \rightarrow \infty} \frac{w'(r)}{w(r)} = \lim {\rm inf}_{r \rightarrow \infty} \beta p'(r) > 1,
$$
since $\lim_{r \rightarrow \infty} \frac{w'(r)}{w(r)} = \lim_{r \rightarrow \infty} \beta p'(r) = \infty$. Moreover, the latter fact implies by (the proof of) \cite[Theorem 3.16]{AT17-2} that every proper closed invariant subspace of $J$ on $H^0_{v_{\beta}}$ has the form $\{f \in H^0_{v_{\beta}} \ | \ f^{(j)}(0)=0, 0 \leq j < K \}$ for some $K \in \N$.

We have $A^0_p(\C) = \mbox{\rm proj}_n H^0_{v_{(1/n)}} $ and the conclusion follows from Theorem \ref{theorem_F} since all the hypothesis of this result hold.

Similarly, since $A_p(\C)= \mbox{\rm ind}_n H^0_{v_{n}} $  the result follows from Theorem \ref{theorem_LB} in this case.
\end{proof}

A different approach enables us to treat the case of the integration operator on the (LB)-algebra $Exp(\C)$ of entire functions of exponential type.

Let $\left(Bf\right)(z) = \frac{f(z)-f(0)}{z}$ denote the backward shift on $H({\mathbb C}).$ We recall that the norm in the Hardy space $H^2$ is given by
$$\|f\|_{H^2} = \sup_{0\leq r < 1}\big(\frac{1}{2\pi}\int_0^{2\pi}\left|f(re^{it})\right|^2\ dt\big)^{\frac{1}{2}}.$$

In what follows we denote
$$
g_\varepsilon (z):= g(\varepsilon z).$$

\begin{lemma}\label{cyclic}
Any transcendent function $f\in H({\mathbb C})$ is a cyclic vector for $B.$
\end{lemma}
\begin{proof}
According to \cite[Theorem 2.2.4]{DSS}, the restriction of any transcendent function to the unit circle is a cyclic vector for the backward shift on the Hardy space $H^2.$ We now fix a compact set $K\subset {\mathbb C}$ and $g\in H({\mathbb C}).$ Take $r > 0$ so that $|z|\leq r$ for all $z\in K$ and put $R = 2r.$ Since
$$
B^n\left(f_R\right) = R^n\left(B^n f\right)_R$$ we have
$$
g_R \in \overline{\mbox{span}}\left(\big(B^n f\right)_R:\ n\in {\mathbb N}\big),$$ where the closure is taken in $H^2.$ So, for every $\varepsilon > 0$ we can find coefficients $a_1, \ldots a_N$ such that
$$
\|g_R - \sum_{k=1}^N a_k \left(B^k f)_R\right)\|_{H^2} < \varepsilon.$$ For every $z\in K$ we have
$$
\begin{array}{*2{>{\displaystyle}l}}
\left|g(z) - \sum_{k=1}^N a_k \left(B^k f\right)(z)\right| & \leq \frac{R}{2\pi}\int_0^{2\pi}\frac{\left|g(Re^{it}) - \sum_{k=1}^N a_k \left(B^k f\right)(Re^{it})\right|}{\left|Re^{it}-z\right|}\ dt \\ & \\ & \leq \frac{2}{2\pi}\int_0^{2\pi}\left|g(Re^{it}) - \sum_{k=1}^N a_k \left(B^k f\right)(Re^{it})\right|\ dt\\ & \\ & \leq 2\|g_R - \sum_{k=1}^N a_k \left(B^k f)_R\right)\|_{H^2} < 2\varepsilon.\end{array}\ $$
\end{proof}
\begin{comment}
Let $J:Exp\ ({\mathbb C})\to Exp\ ({\mathbb C})$ the integration operator
$$
\left(Jf\right)(z) = \int_0^z f(\xi)\ d\xi.$$
\end{comment}

\begin{theorem}\label{exp}
Every proper closed invariant subspace of $J$ on $E = Exp\ ({\mathbb C})$ has the form $A^E_K = \{f\in E:\ f^{j}(0) = 0,\ 0\leq j < K\}$ for some $K\in {\mathbb N}.$
\end{theorem}
\begin{proof}
We consider the topological isomorphism (see for instance \cite[p.94]{T})
$$
\Phi:Exp\ ({\mathbb C})\to H\left({\mathbb C}\right)_b^\prime,\ \langle \Phi(g), f\rangle = \sum_{k=0}^\infty a_k b_k k!,$$ where
$$
f(z) = \sum_{k=0}^\infty a_k z^k,\ g(z) = \sum_{k=0}^\infty b_k z^k.$$ We identify $f\in H({\mathbb C})$ with $T_f\in H({\mathbb C})^{\prime\prime},\ T_f(u):= u(f).$ Then
$$
\Phi^t:H({\mathbb C})\to Exp\ ({\mathbb C})^\prime,\ \langle \Phi^t(T_f),g\rangle = \langle T_f, \Phi(g)\rangle = \Phi(g)(f).$$
\par\medskip
Let us assume that $M$ is a proper closed subspace of $Exp\ ({\mathbb C})$ satisfying $J(M)\subset M.$ Obviously  we also have $J^t(M^\circ)\subset M^\circ.$ We now check that
\begin{equation}\label{eq}
J^t\circ \Phi^t = \Phi^t \circ B.\end{equation} For $f\in H({\mathbb C}), f(z) = \sum_{k=0}^\infty a_k z^k$ and $g\in Exp\ ({\mathbb C}),\ g(z) = \sum_{k=0}^\infty b_k z^k,$ we have
$$
\langle \left(J^t\circ \Phi^t\right)(f), g\rangle = \langle f, \Phi(Jg)\rangle = \sum_{k=1}^\infty a_k\frac{b_{k-1}}{k}k! = \sum_{k=1}^\infty a_k b_{k-1}(k-1)!.$$ Also
$$
\langle \left(\Phi^t\circ B\right)(f), g\rangle = \langle B(f), \Phi(g)\rangle = \sum_{k=0}^\infty a_{k+1}b_kk!.$$ Identity (\ref{eq}) is proved. Consequently
$$
F:= \big(\Phi^t\big)^{-1}\left(M^\circ\right)$$ is a proper closed subspace of $H({\mathbb C})$ and satisfies $B(F)\subset F.$ According to Lemma \ref{cyclic}, $F$ cannot contain transcendent functions. Since $F$ is a closed subspace of the Fr\'echet space $H({\mathbb C})$ and it consists only of polynomials then an application of Baire's theorem permits us to conclude that $F$ is finite dimensional. Let $f_0(z) = \sum_{k=0}^{N-1}a_kz^k + z^N$ be an element of $F$ with the largest possible degree $N$. We observe that $(Bf_0)(z) = \sum_{k=0}^{N-2}a_{k+1}z^k + z^{N-1}.$ Then the polynomials
$$
f_0, Bf_0, B^2f_0,\ldots, B^N f_0$$ are linearly independent and we conclude
$$\mbox{span}\left\{f_0, Bf_0, B^2f_0,\ldots, B^N f_0\right\} = \mbox{span}\left\{1, z, z^2, \ldots, z^N\right\} = F.$$ Finally
$$
\begin{array}{*2{>{\displaystyle}l}}
M = M^{\circ\circ} & = \left\{g\in Exp\ ({\mathbb C})| \ \langle \Phi^t(f), g\rangle = 0\ \forall f\in F\right\} \\ & \\ & = \left\{g\in Exp\ ({\mathbb C})| \ \langle z^k, \Phi(g)\rangle = 0,\ 0\leq k\leq N\right\} \\ & \\ & = \left\{g\in Exp\ ({\mathbb C})| \ g^{(k)}(0) = 0,\ 0\leq k\leq N\right\}.\end{array}$$
\end{proof}

\begin{remark}\label{question}
{\rm The integration operator is continuous on all the H\"ormander algebras defined above $A^0_p(\C)$ and $ A_p(\C)$ by \cite[Lemma 4.1]{BBF15}. However, we do not know whether the conclusion of Theorem \ref{theorem_Ha} holds if the condition $\lim_{r \rightarrow \infty} p'(r) = \infty$ fails, in particular for the weight exponents $p(z)=|z|^s, 0 < s \leq 1$, except in the case covered by Theorem \ref{exp}. }
\end{remark}

\vspace{.3cm}

\textbf{Acknowledgement.} This research was partially supported by the projects  MTM2016-76647-P and GV Prometeo/2017/102.

%%%%%%%%%%%%%%%%%%%%%%%%%%%%%%%%%%%%%%%%%%%%%%%%%%%%%%%%%%%%%%%%%%%%%%%%%
%%%%%%%%%%%%%%%%%%%%%%%%%%%%%%%%%%%%%%%%%%%%%%%%%%%%%%%%%%%%%%%%%%%%%%%%%
%%% Bibliography
%%%%%%%%%%%%%%%%%%%%%%%%%%%%%%%%%%%%%%%%%%%%%%%%%%%%%%%%%%%%%%%%%%%%%%%%%
%%%%%%%%%%%%%%%%%%%%%%%%%%%%%%%%%%%%%%%%%%%%%%%%%%%%%%%%%%%%%%%%%%%%%%%%%

%%%%%%%%%%%%%%%%%%%%%%%%%%%%%%%%%%%%%%%%%%%%%%%%%%%%%%%%%%%%%%%%%%%%%%%%%
%%%%%%%%%%%%%%%%%%%%%%%%%%%%%%%%%%%%%%%%%%%%%%%%%%%%%%%%%%%%%%%%%%%%%%%%%
%%% Other stuff. Otras cosas
%%%%%%%%%%%%%%%%%%%%%%%%%%%%%%%%%%%%%%%%%%%%%%%%%%%%%%%%%%%%%%%%%%%%%%%%%
%%%%%%%%%%%%%%%%%%%%%%%%%%%%%%%%%%%%%%%%%%%%%%%%%%%%%%%%%%%%%%%%%%%%%%%%%

\noindent \textbf{Authors' address:}%
\vspace{\baselineskip}%

Jos\'e Bonet: Instituto Universitario de Matem\'{a}tica Pura y Aplicada IUMPA,
Universitat Polit\`{e}cnica de Val\`{e}ncia,  E-46071 Valencia, Spain

email: jbonet@mat.upv.es \\

Antonio Galbis: Departament d'An\`alisi Matem\`atica,
Universitat de Val\`encia,
E-46100 Burjassot (Val\`encia), Spain

email: antonio.galbis@uv.es \\

\end{document}